\documentclass[11pt]{amsart}
\usepackage{amssymb}

\makeatletter
\newcommand{\sumprime}{\if@display\sideset{}{'}\sum%
            \else\sum'\fi}
\makeatother

\begin{document}

\numberwithin{equation}{section}

% define theorem environments
\newtheorem{theorem}{Theorem}[section]
\newtheorem{proposition}[theorem]{Proposition}
\newtheorem{conjecture}[theorem]{Conjecture}
\def\theconjecture{\unskip}
\newtheorem{corollary}[theorem]{Corollary}
\newtheorem{lemma}[theorem]{Lemma}
\newtheorem{observation}[theorem]{Observation}
\newtheorem{definition}{Definition}
\numberwithin{definition}{section} %\def\thedefinition{\unskip}
\newtheorem{remark}{Remark}
\def\theremark{\unskip}
\newtheorem{question}{Question}
\def\thequestion{\unskip}
\newtheorem{example}{Example}
\def\theexample{\unskip}
\newtheorem{problem}{Problem}

\def\vvv{\ensuremath{\mid\!\mid\!\mid}}
\def\intprod{\mathbin{\lr54}}
\def\reals{{\mathbb R}}
\def\integers{{\mathbb Z}}
\def\N{{\mathbb N}}
\def\complex{{\mathbb C}\/}
\def\dist{\operatorname{dist}\,}
\def\spec{\operatorname{spec}\,}
\def\interior{\operatorname{int}\,}
\def\trace{\operatorname{tr}\,}
\def\cl{\operatorname{cl}\,}
\def\essspec{\operatorname{esspec}\,}
\def\range{\operatorname{\mathcal R}\,}
\def\kernel{\operatorname{\mathcal N}\,}
\def\dom{\operatorname{Dom}\,}
\def\linearspan{\operatorname{span}\,}
\def\lip{\operatorname{Lip}\,}
\def\sgn{\operatorname{sgn}\,}
\def\Z{ {\mathbb Z} }
\def\e{\varepsilon}
\def\p{\partial}
\def\rp{{ ^{-1} }}
\def\Re{\operatorname{Re\,} }
\def\Im{\operatorname{Im\,} }
\def\dbarb{\bar\partial_b}
\def\eps{\varepsilon}
\def\O{\Omega}
\def\Lip{\operatorname{Lip\,}}

\def\Hs{{\mathcal H}}
\def\E{{\mathcal E}}
\def\scriptu{{\mathcal U}}
\def\scriptr{{\mathcal R}}
\def\scripta{{\mathcal A}}
\def\scriptc{{\mathcal C}}
\def\scriptd{{\mathcal D}}
\def\scripti{{\mathcal I}}
\def\scriptk{{\mathcal K}}
\def\scripth{{\mathcal H}}
\def\scriptm{{\mathcal M}}
\def\scriptn{{\mathcal N}}
\def\scripte{{\mathcal E}}
\def\scriptt{{\mathcal T}}
\def\scriptr{{\mathcal R}}
\def\scripts{{\mathcal S}}
\def\scriptb{{\mathcal B}}
\def\scriptf{{\mathcal F}}
\def\scriptg{{\mathcal G}}
\def\scriptl{{\mathcal L}}
\def\scripto{{\mathfrak o}}
\def\scriptv{{\mathcal V}}
\def\frakg{{\mathfrak g}}
\def\frakG{{\mathfrak G}}

\def\ov{\overline}

\thanks{}

\address{}
 \email{}

\title{Supplement to $L^2$ theory of $\bar{\partial}$ on complete K\"ahler domains}
\author{Bo-Yong Chen}
\date{}
\begin{abstract}
  We introduce a trick of dealing with $L^2$ estimates of $\bar{\partial}$ with singular weights on complete K\"ahler domains.
\end{abstract}
\maketitle

\section{Introduction}

A  domain in ${\mathbb C}^n$ is said to be complete K\"ahler if it admits a complete K\"ahler metric. In \cite{Demailly82}, J.-P. Demailly obtained several rather general results which contain as special cases the following H\"ormander type $L^2$ estimate for the $\bar{\partial}-$operator and Skoda type $L^2$ division theorem on complete K\"ahler domains:

\begin{theorem}
Let $\Omega$ be a complete K\"ahler domain in ${\mathbb C}^n$, and $\varphi$ a psh function on $\Omega$ satisfying $i\partial\bar{\partial}\varphi\ge \Theta$ in the sense of distributions for some K\"ahler form $\Theta$ on $\Omega$. For any $\bar{\partial}-$closed $(0,1)$ form $v$ with $\int_\Omega |v|^2 e^{-\varphi}<\infty$ and
$
\int_\Omega |v|^2_\Theta e^{-\varphi}<\infty,
$
there exists $u\in L^2(\Omega,{\rm loc})$ such that $\bar{\partial} u=v$ and
$$
\int_\Omega |u|^2 e^{-\varphi}\le \int_\Omega |v|^2_\Theta e^{-\varphi}.
$$
\end{theorem}

\begin{theorem}
Let $\Omega\subset {\mathbb C}^n$ be a complete K\"ahler domain, and $f\in {\mathcal O}(\Omega)$, $g\in {\mathcal O}(\Omega)^{\oplus m}$, $\varphi\in PSH(\Omega)$. Put $q=\min\{n,m-1\}$. Suppose there exists a number $\alpha>1$ such that
$$
\int_\Omega |f|^2 |g|^{-2(\alpha q+1)} e^{-\varphi}<\infty,
$$
then there exists $h\in {\mathcal O}(\Omega)^{\oplus m}$ satisfying $g\cdot h=f$ and
$$
\int_\Omega |h|^2 |g|^{-2\alpha q} e^{-\varphi}\le {\rm const}_\alpha \int_\Omega |f|^2 |g|^{-2(\alpha q+1)} e^{-\varphi}.
$$
\end{theorem}

The basic difference between a pseudoconvex domain and a complete K\"ahler domain is that only the former can be exhausted by subdomains of same type.  The purpose of this note is to introduce a general trick of dealing with $L^2$ estimates of $\bar{\partial}$ with\/ {\it singular}\/ weights on complete K\"ahler domains,  through giving alternative approaches of the above theorems. The underlying idea goes back to Folland-Kohn \cite{FollandKohn}, Bando \cite{Bando} and Chen-Wu-Wang \cite{ChenWuWang}.

Unfortunately, we could not prove via the same trick the Ohsawa-Takegoshi $L^2$ extension theorem on complete K\"ahler domains, except the special case of $L^2$ extension from a single point, which still has a few amusing applications  (see \cite{ChenWuWang}).

It was pointed out in \cite{XuWang} that this trick actually works for more general situations of holomorphic line bundles with singular Hermitian metrics on complete K\"ahler manifolds (even complete K\"ahler is not necessary), thanks to Demailly's theory of regularization of quasi-psh functions (see e.g. \cite{DemaillyBook10}). Here we stick to the simplest case in order to make the arguments as transparent as possible.

\section{Laplace-Beltrami equation with Dirichlet boundary condition}

Let $\Omega\subset {\mathbb C}^n$ be a bounded domain with smooth boundary and $\omega$ a K\"ahler metric on $\overline{\Omega}$. Let $\varphi$ be a smooth\/ {\it strictly psh}\/ function on $\overline{\Omega}$. Let $D_{(n,k)}(\Omega)$ be the set of smooth $(n,k)-$forms with compact support in $\Omega$, and $L^2_{(n,k)}(\Omega,\varphi)$ the completion of $D_{(n,k)}(\Omega)$ w.r.t. the inner product induced by $\omega$ and $\varphi$. Let $\bar{\partial}_\varphi^\ast$ denote the corresponding formal adjoint of $\bar{\partial}$. Then we have the famous Bochner-Kodaira-Nakano inequality:
\begin{equation}
\|\bar{\partial} u\|^2_\varphi+\|\bar{\partial}^\ast_\varphi u\|_\varphi^2\ge ([i\partial\bar{\partial}\varphi,\Lambda]u,u)_\varphi\ \ \ {\rm for\ all\ }u\in D_{(n,1)}(\Omega).
\end{equation}
Now define an Hermitian form as follows
$$
L(u,v)=(\bar{\partial}u,\bar{\partial}v)_\varphi+(\bar{\partial}^\ast_\varphi u,\bar{\partial}^\ast_\varphi v)_\varphi,\ \ \ u,v\in D_{(n,1)}(\Omega).
$$
 Put $\|u\|_L=\sqrt{L(u,u)}$. By (2.1), we have $\|u\|_L\ge {\rm const.}\|u\|_\varphi$ for all $u\in D_{(n,1)}(\Omega)$, so that $\|\cdot\|_L$ becomes a norm.
 Since $\varphi$ and $\omega$ are smooth on $\overline{\Omega}$, we conclude that $\|\cdot\|_L$ is equivalent to $\|\cdot\|_{W^{1,2}}$ on $D_{(n,1)}(\Omega)$, in view of Garding's inequality (see \cite{FollandKohn}, p. 24).

 Let $H$ denote the Hilbert space of $(n,1)$ forms with coefficients lying in the classical Sobolev space $W^{1,2}_0(\Omega)$ (with respect to the Euclidean metric). Then $(H,\|\cdot\|_L)$ is still a Hilbert space.

 \begin{proposition}
 For any $v\in L^2_{(n,1)}(\Omega,\varphi)$, there exists a unique $w\in H$ such that
 \begin{equation}
 L(u,w)=(u,v)_\varphi\ \ \ {for\ all\ } u\in H
 \end{equation}
 and
 \begin{equation}
  \max\{([i\partial\bar{\partial}\varphi,\Lambda]w,w)_\varphi,\,\|\bar{\partial} w\|^2_\varphi,\,\|\bar{\partial}^\ast_\varphi w\|_\varphi^2\}\le ([i\partial\bar{\partial}\varphi,\Lambda]^{-1} v,v)_\varphi.
 \end{equation}
 \end{proposition}

 \begin{proof}
 Consider the linear functional
 $$
 F(u)=(u,v)_\varphi,\ \ \ u\in H.
 $$
  By the Cauchy-Schwarz inequality, we have
 \begin{eqnarray*}
 |F(u)|^2 & \le & ([i\partial\bar{\partial}\varphi,\Lambda]u,u)_\varphi\, ([i\partial\bar{\partial}\varphi,\Lambda]^{-1} v,v)_\varphi\\
   & \le & ([i\partial\bar{\partial}\varphi,\Lambda]^{-1} v,v)_\varphi\,\|u\|_L^2
 \end{eqnarray*}
    by (2.1), so that there exists a unique $w\in H$ satisfying (2.2), in view of the Riesz representation theorem. Furthermore, we have $\|w\|_L\le 1$ if we assume $([i\partial\bar{\partial}\varphi,\Lambda]^{-1} v,v)_\varphi=1$ for the sake of simplicity. By (2.1), we also have $([i\partial\bar{\partial}\varphi,\Lambda]w,w)_\varphi\le 1$.
 \end{proof}

Put $\Box_\varphi=\bar{\partial}\bar{\partial}^\ast_\varphi+\bar{\partial}^\ast_\varphi\bar{\partial}$. Clearly, we have
$$
L(u,v)=(\Box_\varphi u,v)_\varphi\ \ \ {\rm for\ all\ } u,v\in D_{(n,1)}(\Omega).
$$
Thus the previous proposition essentially gives a (unique) weak solution of the Laplace-Beltrami equation $\Box_\varphi w=v$. Since $\Box_\varphi$ is strongly elliptic, we conclude that $w$ is smooth whenever $v$ is.

\section{Proof of Theorem 1.1}

Let $\tilde{\omega}$ be a complete K\"ahler metric on $\Omega$.
Put
$
\omega=\tilde{\omega}+\Theta.
$
Choose a smooth exhaustion function $\rho$ on $\Omega$ satisfying $|d\rho|_{\tilde{\omega}}\le 1$.
Put $\Omega_j=\{z\in \Omega:\rho(z)<j\}$. For each $j$, we may choose a smooth strictly psh function $\varphi_j$ on $\Omega_{j+1}$ such that $\varphi_j\downarrow \varphi$ as $j\rightarrow \infty$ and $i\partial\bar{\partial}\varphi_j\ge \Theta$ on $\overline{\Omega}_j$.

Let $\chi:{\mathbb R}\rightarrow [0,1]$ be a smooth cut-off function satisfying $\chi|_{(-\infty,1/2)}=1$ and $\chi|_{(1,\infty)}=0$. Assume first that $|\varphi|$ is\/ {\it locally bounded}\/ on $\Omega$. Put $v_j=(\chi(\rho/j)v)\ast \lambda_{\varepsilon_j}$ where $\lambda$ is a standard Friedrichs mollifier and $\varepsilon_j\rightarrow 0$ as $j\rightarrow \infty$. Since $\int_\Omega |v|^2_\Theta e^{-\varphi}<\infty$, so we may choose $\varepsilon_j$ sufficiently small such that
\begin{equation}\label{eq:1}
\int_\Omega |v_j-v|^2_\Theta e^{-\varphi} \rightarrow 0\ \ \ (j\rightarrow \infty).
\end{equation}
Since
$$
\bar{\partial}v_j=(\bar{\partial}(\chi(\rho/j)v))\ast \lambda_{\varepsilon_j}=(\bar{\partial}\chi(\rho/j)\wedge v)\ast \lambda_{\varepsilon_j},
$$
so we get
\begin{equation}\label{eq:2}
\int_{\Omega} |\bar{\partial} v_j|^2_{\omega} e^{-\varphi} \rightarrow 0.
\end{equation}

 For each $(0,k)$ form $u$, we always write $\tilde{u}=dz_1\wedge\cdots\wedge dz_n\wedge u$. Applying Proposition 2.1 to $(\Omega_j,\omega,\varphi_j)$, we get a smooth solution $w_{j}$ of $\Box_{\varphi_j} w=\tilde{v}_j$ together with the following estimate
\begin{equation}
  \max\{\|\bar{\partial} w_{j}\|^2_{\varphi_j},\,\|\bar{\partial}^\ast_{\varphi_j} w_{j}\|_{\varphi_j}^2\}\le ([i\partial\bar{\partial}\varphi_j,\Lambda]^{-1} \tilde{v}_j,\tilde{v}_j)_{\varphi_j}\le \int_\Omega |v_j|^2_\Theta e^{-\varphi}.
 \end{equation}\label{eq:Hormander3}
 Put $\tilde{u}_{j}=\bar{\partial}^\ast_{\varphi_j} w_{j}$ on $\Omega_j$. We may choose a weakly convergent subsequence $\{\tilde{u}_{j_k}\}$ in $L^2_{(n,0)}(\Omega,{\rm loc})$ such that the weak limit $\tilde{u}=u dz_1\wedge\cdots\wedge dz_n$ satisfies
 $$
 \int_\Omega |u|^2 e^{-\varphi} \le  {\lim\inf}_{j\rightarrow \infty}\int_\Omega |v_j|^2_\Theta e^{-\varphi}= \int_\Omega |v|^2_\Theta e^{-\varphi}.
$$
 Since $\tilde{v}_j=\bar{\partial} \tilde{u}_j+\bar{\partial}^{\ast}_{\varphi_j}\bar{\partial} w_j$, and $\tilde{v}_j\rightarrow \tilde{v}$ in the sense of distributions, in view of (\ref{eq:1}) (note that $|\varphi|$ is locally bounded), so $\bar{\partial} \tilde{u}=\tilde{v}$ (i.e., $\bar{\partial}u=v$) if and only if
 \begin{equation}\label{eq:4}
 \bar{\partial}^\ast_{\varphi_j}\bar{\partial} w_j \rightarrow 0
 \end{equation}
 in the sense of distributions. For any $j$, we put $\kappa_j=\chi(\varepsilon/(j-1))$. Clearly,  ${\rm supp\,} \kappa_j\subset \Omega_j$. Since
 $\bar{\partial} \tilde{v}_j=\bar{\partial}\bar{\partial}^\ast_{\varphi_j}\bar{\partial}w_j$ on $\Omega_j$, so we have
 \begin{eqnarray*}
 (\bar{\partial}\tilde{v}_j,\kappa_j^2 \bar{\partial} w_j)_{\varphi_j} = (\bar{\partial}(\kappa_j^2 \bar{\partial}^\ast_{\varphi_j}\bar{\partial}w_j),\bar{\partial} w_j)_{\varphi_j}
 -2 (\kappa_j\bar{\partial} \kappa_j\wedge \bar{\partial}^\ast_{\varphi_j}\bar{\partial}w_j,\bar{\partial}w_j)_{\varphi_j},
 \end{eqnarray*}
 so that
 \begin{eqnarray}\label{eq:5}
 \|\kappa_j \bar{\partial}^\ast_{\varphi_j}\bar{\partial}w_j\|_{\varphi_j}^2 & \le &  \|\bar{\partial}w_j\|_{\varphi_j}(\|\bar{\partial} \tilde{v}_j\|_{\varphi_j}+2\varepsilon_j \sup|\chi'|\|\kappa_j \bar{\partial}^\ast_{\varphi_j}\bar{\partial}w_j\|_{\varphi_j})\nonumber\\
 & \le &  \left(\int_\Omega |v_j|^2_\Theta e^{-\varphi}\right)^{1/2}(\|\bar{\partial} \tilde{v}_j\|_{\varphi_j}+2\varepsilon_j \sup|\chi'|\|\kappa_j \bar{\partial}^\ast_{\varphi_j}\bar{\partial}w_j\|_{\varphi_j}).
 \end{eqnarray}
It follows from $(\ref{eq:1})$ $\sim$ (3.3) that $\|\kappa_j \bar{\partial}^\ast_{\varphi_j}\bar{\partial}w_j\|_{\varphi_j}\rightarrow 0$ as $j\rightarrow \infty$.

 Now let $\Omega'$ be any given relatively compact open subset in $\Omega$. We may choose $j_0$ sufficiently large such that $\Omega'\subset \Omega_j$ and $\kappa_j=1$ on $\Omega'$ for all $j\ge j_0$. Thus for any $f\in D_{(n,1)}(\Omega')$ we have
 \begin{eqnarray*}
 \left| (\bar{\partial}^\ast_{\varphi_j}\bar{\partial}w_j,f)_{\varphi_{j_0}} \right| & = &  \left| (\kappa_j\bar{\partial}^\ast_{\varphi_j}\bar{\partial}w_j,f)_{\varphi_{j_0}} \right|\\
 & \le & \|\kappa_j\bar{\partial}^\ast_{\varphi_j}\bar{\partial}w_j\|_{\varphi_j}\|f\|_{\varphi_{j_0}}\\
 & \rightarrow & 0
 \end{eqnarray*}
 as $j\rightarrow \infty$, so that (\ref{eq:4}) holds.

 For general $\varphi$, we put $\varphi_{m,\varepsilon}=\max\{\varphi,-m\}+\varepsilon |z|^2$, $m=1,2,\cdots$, $\varepsilon>0$. Let
 $$
 \Theta_{m,\varepsilon}:=\chi_{\{\varphi>-m\}}\Theta+\varepsilon i\partial\bar{\partial}|z|^2
 $$
 where $\chi_{\{\varphi>-m\}}$ is the characteristic function on $\{\varphi>-m\}$. Clearly, we have $i\partial\bar{\partial}\varphi_{m,\varepsilon}\ge \Theta_{m,\varepsilon}$.
  For each $m$ and $\varepsilon$, we have a solution $u_{m,\varepsilon}$ of the equation $\bar{\partial}u=v$ satisfying
 $$
 \int_\Omega |u_{m,\varepsilon}|^2 e^{-\varphi_{m,\varepsilon}}\le \int_\Omega |v|_{\Theta_{m,\varepsilon}}^2 e^{-\varphi_{m,\varepsilon}}\le \varepsilon^{-1}\int_{\varphi\le -m}|v|^2 e^{-\varphi} +\int_\Omega |v|_\Theta^2 e^{-\varphi}
 $$
 (here $\Theta_{m,\varepsilon}$ is not continuous, however since $|dz|^2_{\Theta_{m,\varepsilon}}\le 1/\varepsilon$, the argument above still works).
  Since $\int_{\varphi>-m} |v|^2 e^{-\varphi}\rightarrow \int_\Omega |v|^2e^{-\varphi}$ in view of Levi's theorem (notice that $|\varphi^{-1}(-\infty)|=0$), it follows that $\int_{\varphi\le -m}|v|^2 e^{-\varphi}\rightarrow 0$. Let $u_\varepsilon$ be a weak limit of $\{u_{m,\varepsilon}\}_m$ in $L^2(\Omega,{\rm loc})$. Then we have
  $\bar{\partial}u_\varepsilon=v$ and
  $$
  \int_\Omega |u_\varepsilon|^2 e^{-\varphi-\varepsilon |z|^2}\le \int_\Omega |v|_\Theta^2 e^{-\varphi}.
  $$
    It suffices to take a weak limit of $\{u_\varepsilon\}$ in $L^2(\Omega,{\rm loc})$.

\section{Modified Laplace-Beltrami equation with Dirichlet boundary condition}

Let $\Omega$ be a bounded domain with smooth boundary in ${\mathbb C}^n$ and $\omega$ a K\"ahler metric on $\overline{\Omega}$. Let $g\in {\mathcal O}(\overline{\Omega})^{\oplus m}$ with $|g|>0$, and $\varphi\in C^\infty(\overline{\Omega},{\mathbb R})$. Let $D_{(n,k)}(\Omega)$ be the set of smooth $(n,k)-$forms with compact support in $\Omega$, and $L^2_{(n,k)}(\Omega,\varphi)$ the completion of $D_{(n,k)}(\Omega)$ w.r.t. the inner product induced by $\omega$ and $\varphi$.
To solve the division problem it suffices to solve the following vector-valued $\bar{\partial}-$equation
$$
\bar{\partial}u=v:=f dz_1\wedge\cdots\wedge dz_n \wedge \bar{\partial}(\bar{g}/|g|^2)
$$
which satisfies $g\cdot u=0$. Thus it is natural to introduce the following
$$
D_k(\Omega)=\left\{u\in D_{(n,k)}(\Omega)^{\oplus m}: g\cdot u=0 \right\}
$$
and
$$
S_k(\Omega,\varphi)=\left\{u\in L^2_{(n,k)}(\Omega,\varphi)^{\oplus m}: g\cdot u=0 \right\}.
$$
For each $u=(u_1,\cdots,u_m) \in L^2_{(n,k)}(\Omega,\varphi)^{\oplus m}$, we define
$$
\|u\|_\varphi=\sqrt{\sum_k \|u_k\|^2_\varphi}.
$$
Clearly, the completion of  $D_k(\Omega)$ w.r.t. the norm $\|\cdot\|_\varphi$ is contained in $S_k(\Omega,\varphi)$. Yet it is not clear whether they actually\/ {\it coincide}. As $g$ is holomorphic, the $\bar{\partial}$ operator from $D_{(n,0)}(\Omega)^{\oplus m}$ to $D_{(n,1)}(\Omega)^{\oplus m}$ induces a new operator from $D_0(\Omega)$ to $D_1(\Omega)$, which is still denoted by the same symbol for the sake of simplicity. Let $\bar{\partial}^\ast_\varphi$ (resp. $\bar{\partial}^\ast_S$) denote the formal adjoint of $\bar{\partial}:D_{(n,0)}(\Omega)^{\oplus m}\rightarrow D_{(n,1)}(\Omega)^{\oplus m}$ (resp. $\bar{\partial}:D_0(\Omega)\rightarrow D_1(\Omega)$), w.r.t. the inner product $(\cdot,\cdot)_\varphi$. The following crucial observation is essentially due to Ohsawa:

\medskip

\begin{lemma}[cf. \cite{Ohsawa}]
For any $u\in D_1(\Omega)$, we have
$$
\bar{\partial}^\ast_S u=\bar{\partial}^\ast_\varphi u-\bar{g}\cdot \sum_{k=1}^m \bar{\partial}(\bar{g}_k/|g|^2)\lrcorner u_k\footnote{The original formulation in \cite{Ohsawa} is slightly different.}
$$
where $"\lrcorner"$ is the contraction operator.
\end{lemma}

\medskip

\begin{proof}
It is easy to show that the orthogonal complement of $S_0(\Omega,\varphi)$ in $L^2_{(n,0)}(\Omega,\varphi)^{\oplus m}$ is
$$
S_0(\Omega,\varphi)^\bot=\bar{g}\cdot L^2_{(n,0)}(\Omega,\varphi).
$$
Since $L^2_{(n,0)}(\Omega,\varphi)$ is a separable Hilbert space and $D_{(n,0)}(\Omega)$ is dense in $L^2_{(n,0)}(\Omega,\varphi)$, we may choose by the Gram-Schmidt method a complete orthonormal basis $\{e_j\}\subset \bar{g}\cdot D_{(n,0)}(\Omega)$ of $S_0(\Omega,\varphi)^\bot$. Let $P(u)$ be the projection of $\bar{\partial}^\ast_\varphi u$ to $S_0(\Omega,\varphi)$, i.e.,
$$
P(u)=\bar{\partial}^\ast_\varphi u-\sum_j (\bar{\partial}^\ast_\varphi u,e_j)_\varphi e_j.
$$
Put $e_j=\chi_j \bar{g}/|g|$. Clearly $\chi_j\in D_{(n,0)}(\Omega)$ and $\{\chi_j\}$ forms a complete orthonormal basis of $L^2_{(n,0)}(\Omega,\varphi)$. Since $u\in D_1(\Omega)$ and
$$
\bar{\partial} e_j=\bar{\partial}(|g|\chi_j)\cdot \bar{g}/|g|^2+|g|\chi_j\cdot \bar{\partial}(\bar{g}/|g|^2),
$$
it follows that
\begin{eqnarray*}
(\bar{\partial}^\ast_\varphi u,e_j)_\varphi & = & (u,\bar{\partial}e_j)_\varphi=(u,|g|\chi_j\cdot \bar{\partial}(\bar{g}/|g|^2))_\varphi\\
& = & \sum_k (|g|\bar{\partial}(\bar{g}_k/|g|^2)\lrcorner u_k,\chi_j)_\varphi.
\end{eqnarray*}
Thus
\begin{eqnarray*}
P(u) & = & \bar{\partial}^\ast_\varphi u- \sum_j \sum_k (|g|\bar{\partial}(\bar{g}_k/|g|^2)\lrcorner u_k,\chi_j)_\varphi e_j\\
& = & \bar{\partial}^\ast_\varphi u-\bar{g}\cdot \sum_{k=1}^m \bar{\partial}(g_k/|g|^2)\lrcorner u_k,
\end{eqnarray*}
which clearly lies in $D_0(\Omega)$. For any $w\in D_0(\Omega)$, we have
$$
(\bar{\partial}^\ast_S u,w)_\varphi=(u,\bar{\partial} w)_\varphi=(\bar{\partial}^\ast_\varphi u,w)_\varphi=(P(u),w)_\varphi.
$$
Choosing $w=\bar{\partial}^\ast_S u-P(u)$, we immediately get $\bar{\partial}^\ast_S u=P(u)$.
\end{proof}
Put
$$
\Phi_g(u)=\bar{g}\cdot \sum_{k=1}^m \bar{\partial}(\bar{g}_k/|g|^2)\lrcorner u_k.
$$
By Cauchy-Schwarz's inequality, we have
$$
\|\bar{\partial}^\ast_\varphi u\|^2_\varphi \le \frac{\gamma}{\gamma-1}\|\bar{\partial}^\ast_S u\|^2_\varphi+\gamma \|\Phi_g(u)\|^2_\varphi
$$
for all $u\in D_1(\Omega)$ and $\gamma>1$. Combining with the Bochner-Kodaira-Nakano inequality, we obtain
$$
\frac{\gamma}{\gamma-1}\|\bar{\partial}^\ast_S u\|^2_\varphi+\|\bar{\partial}u \|^2_\varphi\ge ([i\partial\bar{\partial}\varphi,\Lambda]u,u)_\varphi-\gamma \|\Phi_g(u)\|^2_\varphi.
$$

Now suppose there is a number $\gamma>1$ such that the RHS of the previous inequality is no less than $\frac{\gamma}{\gamma-1}([\Theta,\Lambda]u,u)_\varphi$ where $\Theta$ is a continuous positive $(1,1)-$form on $\overline{\Omega}$. It follows that
\begin{equation}
\|\bar{\partial}^\ast_S u\|^2_\varphi+\|\bar{\partial} u\|^2_\varphi \ge ([\Theta,\Lambda]u,u)_\varphi\ \ \ {\rm for\ all\ } u\in D_1(\Omega).
\end{equation}
Put $\Box_S=\bar{\partial}\bar{\partial}^\ast_S+\bar{\partial}^\ast_S\bar{\partial}$. Clearly, we have
$$
\|\bar{\partial}^\ast_S u\|^2_\varphi+\|\bar{\partial} u\|^2_\varphi=(\Box_S u,u)_\varphi\ \ \ {\rm for\ all\ } u\in D_1(\Omega).
$$
 Let ${\mathcal H}={\mathcal H}(\Omega,\varphi)$ (resp. ${\mathcal H}_S={\mathcal H}_S(\Omega,\varphi)$) be the completion of $D_1(\Omega)$ in $S_1(\Omega,\varphi)$ w.r.t. the norm $\|\cdot\|_\varphi$ (resp. $\|\bar{\partial}^\ast_S \cdot\|_\varphi+\|\bar{\partial} \cdot\|_\varphi$). Clearly, ${\mathcal H}_S\subset {\mathcal H}$. Similar as \S 2, we may prove the following

\begin{proposition}
For any $v\in {\mathcal H}$, there is a unique weak solution $w\in {\mathcal H}_S$ of the equation $\Box_S w=v$ such that
$$
\max\{([\Theta,\Lambda]w,w)_\varphi,\|\bar{\partial}^\ast_S w\|^2_\varphi,\|\bar{\partial} w\|^2_\varphi\} \le ([\Theta,\Lambda]^{-1}v,v)_\varphi.
$$
\end{proposition}

\section{Proof of Theorem 1.2}

Let us first recall the following

\begin{lemma}[cf. Skoda \cite{Skoda}]
For any matrix $\zeta=(\zeta_{k\mu})_{m\times n}$, we have
$$
q \sum_k \sum_{\mu,\nu} \frac{\partial^2 \log |g|^2}{\partial z_\mu \partial \bar{z}_\nu}\zeta_{k\mu} \bar{\zeta}_{k\nu}\ge |g|^2 \left|\sum_k\sum_\mu \frac{\partial}{\partial z_\mu}(g_k/|g|^2)\zeta_{k\mu}\right|^2.
$$
\end{lemma}

Assume first that $|g|>0$ on $\Omega$. We fix a complete K\"ahler metric $\omega$ on $\Omega$ and take a increasing sequence of smooth subdomains $\Omega_1\subset\subset\Omega_2\subset\subset\cdots$, so that $\Omega=\cup_j \Omega_j$. Choose strictly psh functions $\varphi_j$ on $\overline{\Omega}_j$ such that $\varphi_j\downarrow \varphi$ as $j\rightarrow \infty$. Put
$$
\psi_j=\varphi_j+\alpha q \log |g|^2.
$$
In view of Skoda's lemma, we may choose $\gamma=\frac{\alpha+1}2$ in previous section such that
$$
\|\bar{\partial}^\ast_S u\|^2_{\psi_j}+\|\bar{\partial} u\|^2_{\psi_j}\ge ([\Theta_j,\Lambda]u,u)_{\psi_j}\ \ \ {\rm for\ all\ } u\in D_1(\Omega_j),
$$
 where
$$
\Theta_j=\frac{\alpha-1}{\alpha+1} i\partial\bar{\partial}\left(\varphi_j+\frac{\alpha-1}2q \log |g|^2\right)
$$
is a smooth positive $(1,1)-$form on $\overline{\Omega}_j$. Put $v=f dz_1\wedge\cdots\wedge dz_n\wedge \bar{\partial}(\bar{g}/|g|^2)$. Since
$$
g\cdot v=f dz_1\wedge\cdots\wedge dz_n\wedge \bar{\partial}(g\cdot \bar{g}/|g|^2)=f dz_1\wedge\cdots\wedge dz_n\wedge \bar{\partial} 1=0
$$
and $g\cdot (\kappa v)=0$ for any $\kappa\in C^\infty_0(\Omega_j)$, it follows that $v\in {\mathcal H}_{j}={\mathcal H}(\Omega_j,\psi_j)$. Thus in view of Proposition 4.2 there exists a unique $w_j\in {\mathcal H}_{S,j}={\mathcal H}_S(\Omega_j,\psi_j)$ such that $\Box_S w_j=v$ on $\Omega_j$ and
\begin{eqnarray*}
 && \max\{\|\bar{\partial} w_j\|^2_{\psi_j},\,\|\bar{\partial}^\ast_S w_j\|^2_{\varphi_j}\}\\
 &\le & ([\Theta_j,\Lambda]^{-1}v,v)_{\psi_j} \le  {\rm const}_\alpha \int_\Omega |f|^2 |g|^{-2(\alpha q+1)} e^{-\varphi}
\end{eqnarray*}
where the second inequality follows from Skoda's lemma.  Note also that $w_j$ is smooth on $\Omega_j$ for $\Box_S$ is an elliptic operator. Put $u_j=\bar{\partial}^\ast_S w_j$. Then there is a weakly convergent subsequence $\{u_{j_k}\}$ in $L^2_{(n,0)}(\Omega,{\rm loc})^{\oplus m}$ such that the weak limit $u$ satisfies the following estimate
$$
\int_\Omega |u|^2 |g|^{-2\alpha q} e^{-\varphi}\le {\rm const}_\alpha \int_\Omega |f|^2 |g|^{-2(\alpha q+1)} e^{-\varphi}.
$$
If $\bar{\partial}u=v$ holds in the sense of distributions, then we may conclude the proof by taking
$$
h dz_1\wedge\cdots\wedge dz_n=f \bar{g}/|g|^2 dz_1\wedge\cdots\wedge dz_n-u.
$$
  Since $v=\Box_S w_j=\bar{\partial} u_j+\bar{\partial}^\ast_S \bar{\partial} w_j$, so to verify $\bar{\partial}u=v$ it suffices to show $\bar{\partial}^\ast_S \bar{\partial} w_j\rightarrow 0$ in the sense of distributions. This can be done by a similar argument as \S 3. We include it here for the sake of completeness. Let $\rho,\chi,\kappa_j$ be given as before.  Since $\bar{\partial} \bar{\partial}^\ast_S \bar{\partial} w_j=0$ on $\Omega_j$, so we have
\begin{eqnarray*}
0 & = & (\bar{\partial} \bar{\partial}^\ast_S \bar{\partial} w_j,\kappa_j^2 \bar{\partial} w_j)_{\psi_j}
        =  \|\kappa_j \bar{\partial}^\ast_S \bar{\partial} w_j\|^2_{\psi_j}-2 ( \bar{\partial}^\ast_S  \bar{\partial} w_j,\kappa_j \bar{\partial} \kappa_j \lrcorner\,  \bar{\partial} w_j)_{\psi_j}.
\end{eqnarray*}
It follows from Cauchy-Schwarz's inequality that
$$
\|\kappa_j \bar{\partial}^\ast_S \bar{\partial} w_j\|_{\psi_j}  \le  2 \varepsilon_j\sup|\chi'| \| \bar{\partial}w_j\|_{\psi_j}\le {\rm const}_\alpha\varepsilon_j \int_\Omega |f|^2 |g|^{-2(\alpha q+1)} e^{-\varphi}.
$$
 Now let $\Omega'$ be any given relatively compact open subset in $\Omega$. We may choose $j_0$ sufficiently large such that $\Omega'\subset \Omega_j$ and $\kappa_j=1$ on $\Omega'$ for all $j\ge j_0$. Thus for any $f\in D_{(n,1)}(\Omega')^{\oplus m}$ we have
$$
|(\bar{\partial}^\ast_S  \bar{\partial} w_j,f)_{\psi_{j_0}}|=|(\kappa_j\bar{\partial}^\ast_S \bar{\partial} w_j,f)_{\psi_{j_0}}|\le \|\kappa_j \bar{\partial}^\ast_S \bar{\partial} w_j\|_{\psi_j} \|f\|_{\varphi_{j_0}}\rightarrow 0
$$
as $j\rightarrow \infty$, so that $\bar{\partial}^\ast_S \bar{\partial} w_j\rightarrow 0$ in the sense of distributions.

For general case, we may apply the previous argument to the complete K\"ahler domain $\Omega\backslash g_1^{-1}(0)$ (e.g., $\omega+\partial\bar{\partial}|g_1|^{-2}$ is a complete K\"ahler metric) and conclude the proof by  Riemann's theorem on removable singularities.

\medskip

\textbf{Acknowledgement.}
 The author would like to thank Dr. Xieping Wang for catching an inaccuracy in the previous version of the paper.

\end{document}